\documentclass[12pt,amscd]{amsart}
\usepackage[all]{xy}
\usepackage{graphicx}
\usepackage{amsmath,amsxtra,amssymb,latexsym, amscd,amsthm}
\usepackage{indentfirst}
\usepackage[mathscr]{eucal}
\usepackage[pagebackref=true]{hyperref}
\usepackage{tikz,pgf}
\usepackage{mathrsfs}
\usetikzlibrary{arrows}
\usetikzlibrary{calc,intersections,through,backgrounds}
\usepackage{tkz-euclide}
\numberwithin{equation}{section}

\newtheorem{thm}{Theorem}[section]
\newtheorem{cor}[thm]{Corollary}
\newtheorem{lem}[thm]{Lemma}

\theoremstyle{definition}
\newtheorem{defn}[thm]{Definition}
\newtheorem{exm}[thm]{Example}
\newtheorem{rem}[thm]{Remark}

\newtheorem{conj}{Conjecture}

\DeclareMathOperator{\NN}{\mathbb {N}}
\DeclareMathOperator{\ZZ}{\mathbb {Z}}

\DeclareMathOperator{\lk}{lk}

\DeclareMathOperator{\supp}{supp}
\DeclareMathOperator{\Inter}{Inter}
\DeclareMathOperator{\girth}{girth}
\DeclareMathOperator{\reg}{reg}

\def\ord{\operatorname{ord}}

\def\D {\Delta}

\def\a {\mathbf a}
\def\b {\mathbf b}

\def\m {\mathfrak m}

\def\F {\mathfrak F}

\def\h {\widetilde{H}}

\begin{document}

\title{Regularity of powers of Stanley-Reisner ideals of one-dimensional simplicial complexes}

\author{Nguyen Cong Minh}
\address{Department of Mathematics, Hanoi National University of Education, 136 Xuan Thuy, Hanoi,	Vietnam}
\email{minhnc@hnue.edu.vn}
\author{Thanh Vu}
\email{vuqthanh@gmail.com}


\date{}

\dedicatory{Dedicated to Professor Nguyen Tu Cuong on the occasion of his 70th birthday}
\commby{}
\maketitle
\begin{abstract}
Let $\Delta$ be a one-dimensional simplicial complex. Let $I_\Delta$ be the Stanley-Reisner ideal of $\Delta$. We prove that for all $s \ge 1$ and all intermediate ideals $J$ generated by $I_\Delta^s$ and some minimal generators of $I_\Delta^{(s)}$, we have
$$\reg J = \reg I_\Delta^s = \reg I_\Delta^{(s)}.$$
\end{abstract}

\maketitle

\section{Introduction}\label{sect_intro}
Let $S = K[x_1, ..., x_n]$ be a polynomial ring over a field $K$. Let $\Delta$ be a simplicial complex on $[n]$ that may contain isolated vertices. Via the Stanley-Reisner correspondence, simplicial complexes on $[n]$ are in one-to-one correspondence with squarefree monomial ideals in $S$. For each subset $F$ of $[n]$, let $x_F=\prod_{i\in F}x_i$. The Stanley-Reisner ideal of $\Delta$ is defined by
$$I_\Delta = (x_F \mid  F \notin \Delta).$$

By \cite{CHT,K}, the function $s \mapsto \reg I_\Delta^s$ is asymptotically linear. Nonetheless, for a given simplicial complex $\Delta$, studying the sequence $\{\reg I_\Delta^s| s \ge 1\}$ is a challenging problem. When $\dim \Delta = 1$, Hoa and Trung \cite{HTr} computed $\reg I_\Delta^{(s)}$ for all $s$, while Lu \cite{L} computed the geometric regularity of $I_\Delta^s$. Note that $a_0(I_\Delta^{(s)}) = -\infty$, while Lu did not compute $a_0(I_\Delta^s)$ (see Subsection \ref{subsection_reg} for the definition of the $a_i$-invariants). In this paper, we compute the regularity of all intermediate ideals lying between $I_\Delta^s$ and $I_\Delta^{(s)}$ extending work of Hoa and Trung \cite{HTr} and Lu \cite{L}. More precisely, for monomial ideals $I \subseteq J$, we define $\Inter(I,J)$ the set of monomial ideals $L$ such that $L = I + (f_1, \ldots, f_t)$ where $f_i$ are among minimal monomial generators of $J$. Our main result is:

\begin{thm}\label{dim1} Let $\Delta$ be a one-dimensional simplicial complex. Let $I = I_\Delta$ be the Stanley-Reisner ideal of $\Delta$. Then for all $s \ge 2$ and all intermediate ideal $J \in \Inter(I_\Delta^s,I_\Delta^{(s)})$, we have
$$\reg (J)=\begin{cases}
3s &\text{ if } \girth \Delta = 3,\\ 
2s + 1 &\text{ if } \girth \Delta = 4,\\
2s & \text{ if } \girth \Delta \ge 5.
\end{cases}$$ 
\end{thm}
\begin{exm} Let $\Gamma$ be a one-dimensional simplicial complex whose facets are 
$$\{1,2\}, \{2,3\}, \{3,4\}, \{1,4\}, \{4,5\}, \{5,6\}, \{2,6\}.$$
Let $I = I_\Gamma$. Note that $f_1 = x_2x_3x_4x_5x_6$, $f_2 = x_1x_2x_3x_4x_6$, $f_3 = x_1x_2x_3x_4x_5$ are minimal generators of $I^{(3)}$. Then $I^3, I^3 + (f_1), I^3 + (f_1,f_2), I^3 + (f_1,f_2,f_3)$ all have regularity $7$.
\end{exm}
We noticed the rigidity of regularity for intermediate ideals $J \in \Inter(I^s,I^{(s)})$ for $s = 2,3$ (see \cite{MV} for more detail) from the proofs in our previous work with Nam, Phong, and Thuy \cite{MNPTV}. Prior to this, we did not aware of any rigidity result for the regularity of ideals in non-trivial sequences. We conjecture that this rigidity property holds for edge ideals of graphs.

\begin{conj}\label{conj_rigid} Let $I$ be the edge ideal of a simple graph $G$. For all $s \ge 1$, let $J$ be an intermediate ideal in $\Inter(I^s, I^{(s)})$. Then 
$$\reg (J)=\reg (I^s) =\reg (I^{(s)}).$$
\end{conj}

Theorem \ref{dim1} settles Conjecture \ref{conj_rigid} for edge ideals of graphs with independence number $2$. This paper is a continuation of \cite{MNPTV} where we analyse the degree complex $\Delta_\a(J)$ via the radical ideal $\sqrt{J:x^\a}$. 

We now describe the idea of proof of Theorem \ref{dim1}. An exponent $(\a,i) \in \NN^n \times \NN$ is called an extremal exponent of $J$ if $\reg S/J = |\a| + i$ and $\lk_{\Delta_\a(J)} F$ has a non-vanishing homology in degree $i-1$ for some face $F$ of $\Delta_\a(J)$ such that $F \cap \supp a = \emptyset$. When $\Delta$ contains a triangle, we prove that $x^\a \in J$ if $|\a| \ge 3s$ and $a_i \le s$ for all $i$. We then deduce that $\reg J = 3s$. The main work is in the case where $\Delta$ contains no triangles. Then $I_\Delta = I(G)$ is the edge ideal of $G$, the complement of $\Delta$. By Lemma \ref{reg_lower_bound}, it suffices to prove that 
\begin{equation}\label{eq_upper_bound}
   |\a| +i \le \begin{cases} 2s & \text{ when } \girth \Delta = 4,\\ 2s-1 & \text{ when } \girth \Delta \ge 5. \end{cases} 
\end{equation}
The main tools to analyse $\sqrt{J:x^\a}$, hence $\Delta_\a(J)$ are Lemma \ref{criterion_in_sym} and Lemma \ref{criterion_in_power}. These lemmas provide a new approach toward understanding the behaviour of the sequences $\{\reg I^s \mid s \ge 1\}$ and $\{ \reg I^{(s)} \mid s \ge 1\}$ for general edge ideals. While the description of $\Delta_\a(I^{(s)})$ is known, to our knowledge Lemma \ref{criterion_in_power} gives a first description of $\Delta_\a(I^s)$ for edge ideals of graphs. For a monomial $f \in S$, the $I$-order of $f$, denoted $\ord_I(f)$ is defined by $\ord_I(f) = \max (t \mid  f \in I^t )$.

\smallskip

{\noindent \bf Lemma \ref{criterion_in_power}.}  Let $G$ be a simple graph. Let $I = I(G)$. Let $F$ be an independent set of $G$, and $\a \in \NN^n$ an exponent. Assume that
\begin{equation}
  \sum_{j\in N(F)} a_j + \ord_I(\prod_{u \notin N[F]} x_u^{a_u}) \ge s, \label{eq_in_power0} 
\end{equation}
then $x_F \in \sqrt{I^s:x^\a}$. Conversely, if $x_F$ is a minimal generator of $\sqrt{I^s:x^\a}$ then \eqref{eq_in_power0} holds.

We now outline the steps, where Lemma \ref{criterion_in_sym} and Lemma \ref{criterion_in_power} play a crucial role to establish the upper bound \eqref{eq_upper_bound}. Assume that $\ord_I(x^\a) = r$. A decomposition $x^\a = MN$ is called a good decomposition if $M$ is a minimal monomial generator of $I^r$ and $N \notin I$. An index $i$ where $x_i$ is of highest exponent in $N$ is called a critical index of $\a$ with respect to this decomposition. The steps are:
\begin{enumerate}
    \item We prove that if $i = 2$, then $|\a| \le (\girth \Delta)/(\girth \Delta - 2)(s-1)$, reducing the problem to consider extremal exponents $(\a,i)$ where $|\a| \ge 2s-1$ and $i \le 1$.
    \item Assuming $|\a| \ge 2s-1$, we prove that $\Delta_\a(J) = \Delta_\a(I^s)$.
    \item Assuming $|\a| \ge 2s-1$, let $x^\a = MN$ be a good decomposition of $x^\a$. Let $i$ be a critical index of $N$ with respect to this decomposition. We prove that if $x_i^2|N$, then $\Delta_\a(I^s)$ is a cone over $i$. Consequently, this implies that $|\a| \le 2s$, and if $|\a| = 2s$, then $N$ is squarefree and $\deg N = 2$.
    \item Assuming $|\a| = 2s$, we prove that $\girth \Delta = 4$ and $i = 0$. This completes the case $\girth \Delta = 4$.
    \item Assuming $\girth \Delta \ge 5$ and $|\a| = 2s-1$, we prove that $i = 0$.
\end{enumerate}

One of the main obstructions to carry out this procedure for computing the regularity of intermediate ideals of edge ideals of other classes of graphs is the computation of the order of monomials of the form in Lemma \ref{criterion_in_power}. In subsequent work, we can overcome this difficulty and compute the regularity of intermediate ideals of edge ideals for other classes of graphs.

Now we explain the organization of the paper. In Section \ref{sec_basic}, we recall some notation and basic facts about the symbolic powers of a squarefree monomial ideal, the degree complexes, and Castelnuovo-Mumford regularity. In Section \ref{sec_inter_dim1}, we prove Theorem \ref{dim1}. In Section \ref{sec_small_graphs}, we provide an example where we can verify the rigidity property of the regularity of intermediate ideals for graphs with $\alpha(G) > 2$ to further illustrate our procedure.

\section{Castelnuovo-Mumford regularity, symbolic powers and degree complexes}\label{sec_basic}
In this section, we recall some definitions and properties concerning Castelnuovo-Mumford regularity, the symbolic powers of a squarefree monomial ideal, and the degree complexes of a monomial ideal. The interested reader is referred to (\cite{BH, E, S}) for more details. The material in this section follows closely \cite[Section 2]{MNPTV}.

\subsection{Simplicial complexes and Stanley-Reisner correspondence} 
Let $\Delta$ be a simplicial complex on $[n]=\{1,\ldots, n\}$ that is a collection of subsets of $[n]$ closed under taking subsets. We put $\dim F = |F|-1$, where $|F|$ is the cardinality of $F$. The dimension of $\Delta$ is $\dim \Delta = \max \{ \dim F \mid F \in \Delta \}$.  The set of its maximal elements under inclusion, called by facets, is denoted by $\F(\Delta)$.

A simplicial complex $\D$ is called a cone over $x\in [n]$ if $x\in B$ for any $B\in \F(\Delta)$. If $\D$ is a cone, it is acyclic (i.e., has vanishing reduced homology).

For a face $F\in\Delta$, the link of $F$ in $\Delta$ is the subsimplicial complex of $\Delta$ defined by
$$\lk_{\Delta}F=\{G\in\Delta \mid  F\cup G\in\Delta, F\cap G=\emptyset\}.$$

For each subset $F$ of $[n]$, let $x_F=\prod_{i\in F}x_i$ be a squarefree monomial in $S$. We now recall the Stanley-Reisner correspondence

\begin{defn}For a squarefree monomial ideal $I$, the Stanley-Reisner complex of $I$ is defined by
$$ \Delta(I) = \{ F \subset [n] \mid x_F \notin I\}.$$

For a simplicial complex $\Delta$, the Stanley-Reisner ideal of $\Delta$ is defined by
$$I_\Delta = (x_F \mid  F \notin \Delta).$$
The Stanley-Reisner ring of $\Delta$ is the quotient by the Stanley-Reisner ideal, $K[\Delta] =  S/I_\Delta.$
\end{defn}
From the definition, it is easy to see the following:
\begin{lem}\label{cone} Let $I, J$ be squarefree monomial ideals of  $S = K[x_1,\ldots, x_n]$. Then 
\begin{enumerate}
    \item $\Delta(I)$ is a cone over $t \in [n]$ if and only if $x_t$ is not divided by any minimal generator of $I$.
    \item $I \subseteq J$ if and only if $\Delta(I) \supseteq \Delta(J)$.
\end{enumerate}
\end{lem}

\subsection{Castelnuovo-Mumford regularity}\label{subsection_reg} 
Let $\m = (x_1,\ldots, x_n)$ be the maximal homogeneous ideal of $S = K[x_1,\ldots, x_n]$ a polynomial ring over a field $K$. For a finitely generated graded $S$-module $L$, let
$$a_i(L)=
\begin{cases}
\max\{j\in\ZZ \mid H_{\m}^i(L)_j \ne 0\} &\text{ if  $H_{\m}^i(L)\ne 0$}\\ 
-\infty &\text{ otherwise,}
\end{cases}
$$
where $H^{i}_{\m}(L)$ denotes the $i$-th local cohomology module of $L$ with respect to $\m$. Then, the Castelnuovo-Mumford regularity (or regularity for short) of $L$ is defined to be
$$\reg(L) = \max\{a_i(L) +i\mid i = 0,\ldots, \dim L\}.$$

For a non-zero and proper homogeneous ideal $J$ of $S$ we have $\reg(J)=\reg(S/J)+1$.

\subsection{Graphs and their edge ideals}

Let $G$ denote a finite simple graph over the vertex set $V(G)=[n] = \{1,2,\ldots,n\}$ and the edge set $E(G)$. For a vertex $x\in V(G)$, let the neighbour of $x$ be the subset $N_G(x)=\{y\in V(G)~|~ \{x,y\}\in E(G)\}$, and set $N_G[x]=N_G(x)\cup\{x\}$. For a subset $U$ of the vertices set $V(G)$, $N_G(U)$ and $N_G[U]$ are defined by $N_G(U)=\cup_{u\in U}N_G(u)$ and $N_G[U]=\cup_{u\in U}N_G[u]$. If $G$ is fixed, we shall use $N(U)$ or $N[U]$ for short.

An independent set in $G$ is a set of pairwise non-adjacent vertices. A maximal independent set is an independent set that is maximal under inclusion. The independence number of $G$, denoted by $\alpha(G)$, is the cardinality of a maximal independent set of maximum size. 

A subgraph $H$ is called an induced subgraph of $G$ if for any vertices $u,v\in V(H)\subseteq V(G)$ then $\{u,v\}\in E(H)$ if and only if $\{u,v\}\in E(G)$.

An induced matching is a subset of the edges that do not share any vertices and it is an induced subgraph. The induced matching number of $G$, denoted by $\mu(G)$, is the largest size of an induced matching in $G$.

A $m$-cycle in $G$ is a sequence of $m$ distinct vertices $1,\ldots, m\in V(G)$ such that $\{1,2\},\ldots, \{m-1,m\}, \{m,1\}$ are edges of $G$. The girth of $G$, denoted $\girth (G)$ is the size of a smallest induced cycle in $G$.

The edge ideal of $G$ is defined to be
$$I(G)=(x_ix_j~|~\{i,j\}\in E(G))\subseteq S.$$
For simplicity, we often write $i \in G$ (resp. $ij \in G$) instead of $i \in V(G)$ (resp. $\{i,j\} \in E(G)$).

A clique in $G$ is a complete subgraph of $G$. We also call a clique of size $3$ a triangle. 

\subsection{Symbolic powers} 
Let $I$ be a non-zero and proper homogeneous ideal of $S$. Let $\{P_1,\ldots,P_r\}$ be the set of the minimal prime ideals of $I$. Given a positive integer $s$, the $s$-th symbolic power of $I$ is defined by
$$I^{(s)}=\bigcap_{i=1}^r I^sS_{P_i}\cap S.$$

For a monomial $f$ in $S$, we denote $\frac{\partial^* (f)}{\partial^*(x^\a)}$ the $*$-partial derivative of $f$ with respect to $x^\a$, which is derivative without coefficients. In general, $\partial f / \partial x^\a = c \partial^*(f) / \partial^*(x^\a)$ for some constant $c$. We define 
$$I^{[s]} =  ( f \in S ~|~ \frac{\partial^* f }{\partial^* x^\a} \in I, \text{ for all } x^\a \text{ with } |\a| = s -1),$$ the $s$-th $*$-differential power of $I$. When $I$ is a squarefree monomial ideal, the symbolic powers of $I$ is equal to the $*$-differential powers of $I$. 

\begin{lem}\label{differential_criterion} Let $I$ be a squarefree monomial ideal. Then $I^{(s)} = I^{[s]}.$
\end{lem}

For a monomial $f$ in $S$, and $i \in [n]$, $\deg_i(f) = \max(t \mid  x_i^t \text{ divides }f)$ denotes the degree of $x_i$ in $f$. The support of $f$, denoted $\supp(f)$, is the set of all indices $i \in [n]$ such that $x_i|f$. The radical of $f$ is defined by $\sqrt{f} = \prod_{i \in \supp f} x_i$. As a consequence, we deduce the following property of generators of $I^{(s)}$.

\begin{lem}\label{partial_deg_bound_1} Let $I$ be a squarefree monomial ideal. Let $f$ be a minimal generator of $I^{(s)}$. Then $\deg_i(f) \le s$ for all $i = 1, ..., n$.
\end{lem}
\begin{proof}
 Assume by contradiction that $f=x_1^tg$ is a minimal generator of $I^{(s)}$ where $t >s$ and $x_1 \not | g$. Since $f$ is minimal, $f_1 = x_1^{t-1}g$ does not belong to $I^{(s)}$. By Lemma \ref{differential_criterion}, there exists an exponent $\a$ such that $|\a | = s-1$ and 
 $$h_1 = \partial^* (f_1)/\partial^*(x^\a) \notin I.$$
 Since $|\a| = s-1$ and $t-1 \ge s$, $x_1 | h_1$. This implies that 
 $$h = \partial^*(f)/\partial^*(x^\a) = x_1h_1.$$
 By Lemma \ref{differential_criterion}, $h \in I$, but this is a contradiction, as $I$ is squarefree and $\sqrt{h} = \sqrt{h_1}$. The conclusion follows.
\end{proof}

As another consequence, we give a necessary condition for a squarefree monomial to be in $\sqrt{J:x^\a}$.

\begin{lem}\label{criterion_in_sym} Let $I$ be a squarefree monomial ideal. Let $J \in \Inter(I^s,I^{(s)})$ be an intermediate ideal lying between $I^s$ and $I^{(s)}$. Let $\a \in \NN^n$ be an exponent such that $x^\a \notin J$. Assume that $f \in \sqrt{J : x^\a}$ and that $f \notin I$. Let $F$ be a facet of $\Delta(I)$ that contains $\supp (f)$. Then $\sum_{i \notin F} a_i \ge s$.
\end{lem}
\begin{proof} Assume by contradiction that $\sum_{i \notin F} a_i \le s-1$. Let $\b \in \NN^n$ be an exponent such that $b_i = 0$ for all $i \in F$, $b_i = a_i$ for all $i\notin F$. Then $|\b| \le s-1$. Furthermore,  
	$$\supp \big ( \dfrac{\partial^* (f^u x^\a)}{\partial^* (x^\b)} \big ) \subseteq F$$
	implying that $f^u x^\a \notin I^{(s)}$ for all $u \ge 1$ by Lemma \ref{differential_criterion}, which is a contradiction to the fact that $f \in \sqrt{J:x^\a} \subseteq \sqrt{I^{(s)}:x^\a}.$
\end{proof}

\subsection{Degree complexes}
For a monomial ideal $I$ in $S$, Takayama in \cite{T} found a combinatorial formula for $\dim_KH_\m^i(S/I)_\a$ for all $\a\in\ZZ^n$ in terms of certain simplicial complexes which are called degree complexes. For every $\a = (a_1,\ldots, a_n) \in \ZZ^n$ we set $G_\a = \{i\mid \ a_i < 0\}$ and write $x^{\a} = \Pi_{j=1}^n x_j^{a_j}$. Thus, $G_\a =\emptyset$ whenever $\a \in \NN^n$. The degree complex $\D_\a(I)$ is the simplicial complex whose faces are $F \setminus G_\a$, where $G_\a\subseteq F\subseteq [n]$, so that for every minimal generator $x^\b$ of $I$ there exists an index $i \not\in F$ with $a_i < b_i$. It is noted that $\D_\a(I)$ may be either the empty complex or $\{\emptyset\}$ and its vertex set may be a proper subset of $[n]$. The next lemma is useful to compute the regularity of a monomial ideal in terms of its degree complexes.

\begin{lem}\label{Key0}
Let $I$ be a monomial ideal in $S$. Then
\begin{multline*}
\reg(S/I)=\max\{|\a|+i~|~\a\in\NN^n,i\ge 0,\h_{i-1}(\lk_{\D_\a(I)}F;K)\ne 0\\ \text{ for some $F\in \D_\a(I)$ with $F\cap \supp \a=\emptyset$}\}.
\end{multline*}
In particular, if $I=I_\D$ is the Stanley-Reisner ideal of a simplicial complex $\D$ then
$\reg(K[\D])=\max\{i~|~i\ge 0,\h_{i-1}(\lk_{\D}F;K)\ne 0\text{ for some }F\in \D\}$.
\end{lem}

\begin{rem}\label{rem_T} Let $I$ be a monomial ideal in $S$ and a vector $\a\in \NN^n$. In the proof of \cite[Theorem 1]{T}, the author showed that if there exists $j\in [n]$ such that $a_j\ge \rho_j=\max\{\deg_{j}(u)~|~ u\text{ is a minimal monomial generator of } I\}$ then $\D_\a(I)$ is either a cone over $\{j\}$ or the void complex. 
\end{rem}

\begin{defn}\label{exdef} Let $I$ be a monomial ideal in $S$. A pair $(\a,i) \in \NN^n\times\NN$ is called {\it an extremal exponent of the ideal $I$}, if $\reg(S/I) = |\a| + i$ as in Lemma \ref{Key0}. 
\end{defn}

\begin{rem}\label{rem_extremal_set} We sometime call $\a$ instead of $(\a,i)$ an extremal exponent of $I$. Let $\a$ be an extremal exponent of $I$. Then $x^\a \notin I$ and $\Delta_{\a}(I)$ is not a cone over $t$ with $t\in\supp\a$. In particular, by Remark \ref{rem_T}, $\a$ belongs to the finite set
	$$\Gamma(I)=\{\a\in\NN^n~|~ a_j<\rho_j\text{ for all } j=1,\ldots,n\}.$$ 
\end{rem}

Furthermore, we have the following interpretation of the degree complex $\Delta_\a(I)$.

\begin{lem}\label{Key1}
Let $I$ be a monomial ideal in $S$ and $\a\in\NN^n$. Then
$$I_{\Delta_{\a}(I)}=\sqrt{I : x^\a}.$$
\end{lem}
\begin{rem}\label{rem_mingens_degree_complex} Let $\a$ be an extremal exponent of $I$. By Lemma \ref{cone}, Lemma \ref{Key1}, and Remark \ref{rem_extremal_set}, for each $t \in \supp \a$, there exists a minimal generator $f$ of $\sqrt{I:x^\a}$ such that $x_t | f$.
\end{rem}

 We first deduce the following inequality on the regularity of restriction of a monomial ideal.

\begin{lem}\label{restriction_in} Let $I$ be a monomial ideal and $x_j$ is a variable. Then 
$$\reg (I,x_j) \le \reg I.$$
\end{lem}

For an exponent $\a \in \ZZ^n$, we denote $\supp(\a) =\{i\in [n] \mid \ a_i \neq 0\}$, the support of $\a$. For any subset $V \subset [n]$, we denote $$I_V = (f~|~  f \text{ is a monomial which belongs to } I \text{ and  } \supp (f) \subseteq V)$$ be the restriction of $I$ on $V$. We have

\begin{cor}\label{restriction_inq} Let $J$ be a monomial ideal in $S$. Let $V \subseteq [n]$. We have
$$\reg (J_V) \le \reg (J).$$
\end{cor}
\begin{proof}
 Let $\{t,\ldots,n\} = [n] \setminus V$. Then, $J_V + (x_t,...,x_n) = J + (x_t,...,x_n).$ The conclusion follows from Lemma \ref{restriction_in} and the fact that $x_t,\ldots, x_n$ is a regular sequence with respect to $S/J_V$.
\end{proof}

Consequently, we deduce a lower bound on the regularity of intermediate ideals of edge ideals.

\begin{cor}\label{reg_lower_bound} Let $I = I(G)$ be the edge ideal of a simple graph $G$. Let $J \in \Inter(I^s,I^{(s)})$ be an intermediate ideal. Then $\reg J \ge 2s + \mu(G) - 1$.
\end{cor}
\begin{proof} Let $H$ be a maximum induced matching of $G$. Then $I_H^s \subseteq J_H \subseteq I_H^{(s)}$. Since $I_H$ is a complete intersection of $\mu(G)$ quadrics, $I_H^{(s)} = I_H^s$ and $\reg I_H^s = 2s + \mu(G) - 1$. The Corollary follows from Corollary $\ref{restriction_inq}$ and \cite[Corollary 2.4]{MNPTV}.
\end{proof}

\subsection{Radicals of colon ideals} We start with a simple observation.

\begin{lem}\label{radical_colon} Let $I$ be a monomial ideal in $S$ generated by the monomials $f_1, ..., f_r$ and $\a \in \NN^n$. Then $\sqrt{I:x^\a}$ is generated by $\sqrt{f_1/\gcd(f_1, x^\a)}, ..., \sqrt{f_r/\gcd(f_r,x^\a)}$.
\end{lem}

Assume now that $I = I(G)$ is the edge ideal of a graph, and $\a\in \NN^n$ an exponent. We will give a sufficient condition for a squarefree monomial to belong to the radical $\sqrt{I^s:x^\a}$. It is also necessary for $f$ to be minimal. First, we introduce some notation. Let $f$ be a monomial. The $I$-order of $f$ is defined by 
$$\ord_I(f) = \max (t \mid f \in I^t).$$
From the definition, it is clear that if $g|f$, then $\ord_I(g) \le \ord_I(f)$.

\begin{lem}\label{criterion_in_power}  Let $F$ be an independent set of $G$, and $\a \in \NN^n$ an exponent. Assume that
\begin{equation}
  \sum_{j\in N(F)} a_j + \ord_I(\prod_{u \notin N[F]} x_u^{a_u}) \ge s, \label{eq_in_power} 
\end{equation}
then $x_F \in \sqrt{I^s:x^\a}$. Conversely, if $x_F$ is a minimal generator of $\sqrt{I^s:x^\a}$ then \eqref{eq_in_power} holds.
\end{lem}
\begin{proof} Let 
$$g = \prod_{j\in N(F)} x_j^{a_j} \text{ and } h = \prod_{u \notin N[F]} x_u^{a_u}.$$
We have $x_F^{\deg g} g \in I^{\deg g}$. Therefore, if $\deg g + \ord_I(h) \ge s$, $x_F^{\deg g} gh \in I^s$.

Conversely, assume that $x_F \in \sqrt{I^s:x^\a}$ is a minimal generator. By Lemma \ref{radical_colon}, there exist edges $e_1, ..., e_s$ of $G$ such that $x_F = \sqrt{P/\gcd(P,x^\a)}$ where $P = e_1 \cdots e_s$. Replacing $x^\a$ by $\gcd(P,x^\a)$, we may assume that there exists exponents $b_i > 0$, $i \in F$ such that 
$$P = (\prod_{i \in F} x_i^{b_i}) \cdot x^\a.$$
Assume that among $e_1, ..., e_s$ there are $s_1$ edges of the form $x_ix_j$, $s_2$ edges of the form $x_jx_k$, $s_3$ edges of the form $x_k x_u$, and $s_4$ edges of the form $x_u x_v$ where $i \in F$, $j,k \in N(F)$ and $u,v \notin N[F]$. Then $s = s_1 + s_2 + s_3 +s_4$ and 
    $$\sum_{j \in N(F)} a_j  =        s_1 + 2s_2 + s_3.$$
Furthermore, $h$ is divisible by the product of $s_4$ edges of the form $x_ux_v$. In particular, $\ord_I(h) \ge s_4$. Thus
    $$\sum_{j\in N(F)} a_j + \ord_I(h) \ge  s_1 + 2s_2 + s_3 + s_4 \ge s,$$
    as required.
\end{proof}

As a consequence, we have

\begin{cor}\label{cor_maximal_independent_term} Let $F$ be a maximal independent set of $G$. Let $J \in (I^s,I^{(s)})$ be an intermediate ideal. Let $\a\in \NN^n$ be an exponent such that $x^\a \notin J$. Then $x_F \in \sqrt{J:x^\a}$ if and only if $x_F \in \sqrt{I^s:x^\a}$.
\end{cor}
\begin{proof}Since $\sqrt{I^s:x^\a} \subseteq \sqrt{J:x^\a}$, it suffices to prove that if $x_F \in \sqrt{J:x^\a}$ then $x_F \in \sqrt{I^s:x^\a}$. Since $F$ is a maximal independent set $F \cap N(F) = \emptyset$ and $N[F] = F \cup N(F) = [n]$. By Lemma \ref{criterion_in_sym}, 
$$\sum_{j \notin F} a_j \ge s.$$
The conclusion follows from Lemma \ref{criterion_in_power}.
\end{proof}

From Lemma \ref{criterion_in_power}, we see that the computation of order is of particular importance. The order of a clique term, which plays a crucial role in the study of $\Delta_\a(I^s)$ when $\alpha(G) = 2$, can be computed as follows.

\begin{lem}\label{clique_term} Assume that $\{1, ..., t\}$ is a clique of $G$. Let $f = x_1^{a_1} \cdots x_t^{a_t}$. Assume that $a_1 = \max(a_1, ..., a_t)$. Then
$$\ord_I(f) = \min(a_2+ \cdots + a_t, \lfloor (a_1 + \cdots + a_t)/2 \rfloor).$$
\end{lem}
\begin{proof} We may assume that $a_1 \ge a_2 \ge \cdots \ge a_n$. Let $r = \min(a_2+ \cdots + a_t, \lfloor (a_1 + \cdots + a_t)/2 \rfloor).$ We prove by induction on $r$ that $f \in I^r$. The base case $r = 0$ is obvious.

Assume that $r \ge 1$, then $a_2 \ge 1$. Let $v$ be the largest index such that $a_v > 0$, and $u$ be the largest index such that $a_u = a_1$. In other words, we have $a_1 = \cdots = a_u > a_{u+1} \ge \cdots a_v > a_{v+1} = 0$. If $u = v$, i.e. $a_1 = \cdots = a_v$, then let $\b = (a_1, ...,a_{v-1}-1,a_v-1)$, then $f = x^\b (x_{v-1}x_v)$ and by induction $x^\b \in I^{r-1}$. If $u >v$, let $\b = (a_1, ..., a_{u-1}, a_u-1, a_{u+1}, ...,a_v-1)$. Then $f = x^\b (x_ux_v)$ and $x^\b \in I^{r-1}$ by induction.

We will now show that $f \notin I^{r+1}$. Indeed, if $a_1 \le \sum_{i = 2}^ta_i$, then $\deg f \le 2r + 1$, and thus $f \notin I^{r+1}$. Assume that $a_1 > a_2 + \cdots  + a_t$. Let $\b = (0, a_2, ..., a_t)$. Then $|\b| = r = a_2 + \cdots + a_t$. We have $\partial^* f/ \partial^* x^\b = x_1^{a_1} \notin I$. Thus $f \notin I^{(r+1)}$. In particular $f \notin I^{r+1}$ as required.  
\end{proof}

\section{Proof of Theorem \ref{dim1}}\label{sec_inter_dim1}

Throughout the section, $\Delta$ denotes a one-dimensional simplicial complex, $I = I_\Delta$ the Stanley-Reisner ideal of $\Delta$. Since $\dim \Delta = 1$, $\reg I_\Delta \le 3$ by Lemma \ref{Key0}. Thus, $I$ is generated by squarefree monomials of degree at most three. Furthermore, $I$ has a minimal generator of degree three if and only if $\Delta$ contains a triangle. We first give a simple proof of Theorem \ref{dim1} when $\girth \Delta = 3$.

\begin{lem}\label{degree_bound} Let $\Delta$ be a one-dimensional simplex such that $\girth \Delta = 3$. Let $I = I_\Delta$. Let $\a\in \NN^n$ be an exponent such that $a_i \le s$ for all $i=1, ..., n$. Assume that $|\a| \ge 3s$, then $x^\a \in I^s$.
\end{lem}
\begin{proof} We prove by induction on $s$. The base case when $s = 1$ follows from the fact that $\dim \Delta = 1$. 

Assume that $s \ge 2$. We may assume that $a_1 \ge a_2 \ge \cdots \ge a_n$. If $a_3 =s$, then $a_1=a_2 = s$ by assumption. Thus $x^\a$ is divisible by $(x_1x_2x_3)^s$. Thus $x^\a \in I^s$. Thus we may assume that $a_i \le s-1$ for all $i \ge 3$. 

Furthermore, since $a_1 + a_2 \le 2s$ and $|\a| \ge 3s$, we deduce that $a_3 \ge 1$. Let $\b=(a_1-1,a_2-1,a_3-1,a_4, ..., a_n)$. Then $b_i \le s-1$ for all $i$ and $|\b| = |\a| - 3 \ge 3(s-1)$. By induction $x^\b \in I^{s-1}$. Therefore, $x^\a = (x_1x_2x_3)x^\b \in I^s$.
\end{proof}

\begin{lem}\label{reg_girth_3} Let $\Delta$ be a one-dimensional simplex such that $\girth \Delta = 3$. Let $I = I_\Delta$. For $s\ge 1$, let $J \in \Inter(I^s,I^{(s)})$ be an intermediate ideal lying between $I^s$ and $I^{(s)}$. Then $\reg J = 3s$. 
\end{lem}
\begin{proof} Since $\girth \Delta = 3$, $I$ has a generator of degree three. By degree reason, $\reg J \ge 3s$. Let $(\a,i)$ be an extremal exponent of $J$. By Remark \ref{rem_extremal_set} and Lemma \ref{partial_deg_bound_1}, $a_i \le s-1$ for all $i = 1, ..., n$ and $x^\a \notin J$. By Lemma \ref{degree_bound}, $|\a| \le 3s-1$. It suffices to show that $|\a| +i \le 3s-1$. By Lemma \ref{cone}, $\Delta_\a(J)$ is a subsimplicial complex of $\Delta$, thus $i \le 2$. It remains to consider the cases where $|\a| = 3s-1$ and $|\a| = 3s-2$. 

If $|\a| = 3s - 1$. By Lemma \ref{degree_bound}, $x_jx^\a \in I^s$, for all $j \in [n]$, i.e. $I^s:x^\a = (x_1,...,x_n)$. Thus $i = 0$ and $\reg J = 3s$. 

If $|\a| = 3s - 2$. By Lemma \ref{degree_bound}, $x_u x_v  x^\a \in I^s$ for all $u \neq v \in [n]$. By Lemma \ref{cone}, $\Delta_\a(J)$ is a subsimplicial complex of the zero-dimensional simplicial complex whose facets are $\{1\}, ..., \{n\}$. Therefore, $i \le 1$, and $|\a| + i \le 3s-1$, as required.
\end{proof}

Now assume that $\girth \Delta \ge 4$. In this case, $I_\Delta$ is the edge ideal of the complement of $\Delta$. Till the rest of the section, we denote $G$ the complement of $\Delta$. Then $G$ has no isolated vertices, $\alpha(G) = 2$ and $I = I(G)$ is the edge ideal of $G$. Before proving Theorem \ref{dim1}, we recall the steps from Introduction for convenient of readers. Fix $s \ge 2$ and $J \in \Inter(I^s,I^{(s)})$ an intermediate ideal. By Corollary \ref{reg_lower_bound}, 
$$\reg J \ge 2s + \mu(G) - 1 = \begin{cases} 2s + 1 & \text{ when } \girth \Delta = 4 \\ 2s & \text{ when } \girth \Delta \ge 5.\end{cases}$$
Let $(\a,i)$ be an extremal exponent of $J$. Assume that $\ord_I(x^\a) = r$. A decomposition $x^\a = MN$ is called a good decomposition if $M$ is a minimal monomial generator of $I^r$, and $N \notin I$. An index $i$ where $x_i$ is of highest exponent in $N$ is called a critical index of $\a$ with respect to this decomposition. The steps to prove the required upper bounds for $|\a| + i$ are:

\begin{enumerate}
    \item We prove that if $i = 2$, then $|\a| \le (\girth \Delta)/(\girth \Delta - 2)(s-1)$, reducing the problem to consider extremal exponents $(\a,i)$ where $|\a| \ge 2s-1$ and $i \le 1$.
    \item Assuming $|\a| \ge 2s-1$, we prove that $\Delta_\a(J) = \Delta_\a(I^s)$.
    \item Assuming $|\a| \ge 2s-1$, let $x^\a = MN$ be a good decomposition of $x^\a$. Let $i$ be a critical index of $N$ with respect to this decomposition. We prove that if $x_i^2|N$, then $\Delta_\a(I^s)$ is a cone over $i$. Consequently, this implies that $|\a| \le 2s$, and if $|\a| = 2s$, then $N$ is squarefree and $\deg N = 2$.
    \item Assuming $|\a| = 2s$, we prove that $\girth \Delta = 4$ and $i = 0$. This completes the case $\girth \Delta = 4$.
    \item Assuming $\girth \Delta \ge 5$ and $|\a| = 2s-1$, we prove that $i = 0$.
\end{enumerate}

We fix the following notation throughout the rest of the section. Let $\Delta$ be a one-dimensional simplicial complex with $\girth \Delta \ge 4$. Let $G$ be the complement of $\Delta$. Let $I = I_\Delta = I(G)$, and $J \in \Inter(I^s,I^{(s)})$ an intermediate ideal.

We now proceed to the first step to bound the degree of extremal exponents $(\a,i)$ of $J$ when $i = 2$.

\begin{lem}\label{h1_bounds} Assume that $(\a,2)$ is an extremal exponent of $J$. Let $\ell$ be the length of a smallest cycle in $\Delta_\a(J)$. Then 
$$|\a| \le \ell (s-1)/(\ell -2).$$
In particular, if $\girth \Delta =4$, then $|\a| \le 2s-2$, and if $\girth \Delta \ge 5$, then $|\a| \le 2s-3$.
\end{lem}
\begin{proof}
    Let $1\cdots \ell$ be a cycle in $\Delta_\a(J)$. Then we have $x_ix_{i+1} \notin \sqrt{J:x^\a}$. In particular $x_ix_{i+1} \notin I$. But then for every $t \notin \{i,i+1\}$, $t$ is adjacent to either $i$ or $i+1$ as $\Delta$ is of dimension $1$. By Lemma \ref{criterion_in_power}, we have 
    $$|\a| - a_i - a_{i+1} \le s-1.$$
    Summing over all $\ell$ edges we deduce that 
    $$(\ell-2) |\a| \le \ell(s-1).$$
    In particular, 
    $$|\a| \le \frac{\ell (s-1)}{\ell-2}.$$
    By Lemma \ref{cone}, and the fact that $I \subseteq \sqrt{J:x^\a}$, $\Delta_\a(J)$ is a subsimplicial complex of $\Delta$. Thus $\ell \ge \girth \Delta$, and the last past follows immediately.
\end{proof}

To accomplish step 2, we also need the following property of minimal generators of symbolic powers. 

\begin{lem}\label{partial_deg_bound_2} Let $G$ be a simple graph such that $\alpha(G) = 2$. Let $I = I(G)$ be the edge ideal of $G$. Let $f=x^\a$ be a minimal monomial generator of $I^{(s)}$. Then 
$$a_i \le \sum_{j\in N(i)} a_j.$$
\end{lem}
\begin{proof} We may assume that $i = 1$ and $N(1) = \{2, ..., t\}$. Assume by contrary that $a_1  > a_2 + \cdots + a_t$. Let $g=f/x_1$. Then $g \notin I^{(s)}$ and $x_1 \in \sqrt{I^{(s)}:g}$. Let $h = x_{t+1}^{a_{t+1}} \cdots x_n^{a_n}$ and we may assume that $a_{t+1} \ge \cdots \ge a_n$. Since $\{t+1, ..., n\}$ forms a clique in $G$, $h \in I^{(s_2)}$, where $s_2 = a_{t+2} + \cdots + a_n$. By Lemma \ref{criterion_in_sym}, $\sum_{i\neq 1, t+1} a_i  \ge s$. Now $x_1^{a_1-1} x_2^{a_2} \cdots x_t^{a_t} \in I^{s_1}$ where $s_1 = a_2 + \cdots + a_t$ as $a_1 - 1 \ge a_2 + \cdots + a_t$. Therefore, $g \in I^{(s_1 + s_2)} \subseteq I^{(s)}$, which is a contradiction.
\end{proof}

We are now ready for step 2.

\begin{lem}\label{intermediate_reduction} Let $\a$ be an exponent such that $x^\a \notin J$. Assume that $|\a| \ge 2s-1$, then $\Delta_\a(J) = \Delta_\a(I^s)$.
\end{lem}
\begin{proof} Assume that $\sqrt{J:x^\a} \neq \sqrt{I^s:x^\a}$. Let $g$ be a minimal generator of $\sqrt{J:x^\a}$ such that $g \notin \sqrt{I^s:x^\a}$. Therefore, $g \notin I$. By Corollary \ref{cor_maximal_independent_term} and the fact that $\dim \Delta = 1$, we deduce that $\deg g= 1$. We may assume that $g=x_1$. Let $\{2,3,...,t\}$ be the set of neighbors of $1$. Assume that $a_{t+1} \ge a_{t+2} \ge \cdots \ge a_n$. By Lemma \ref{criterion_in_sym}, 
	$$ a_2 + \cdots + a_t + a_{t+2} + \cdots + a_n \ge s.$$
	Since $x_1 \notin \sqrt{I^s:x^\a}$, by Lemma \ref{criterion_in_power} and Lemma \ref{clique_term}, we deduce that 
	$$ a_2 + a_3 + \cdots +a_t + \lfloor (a_{t+1} + \cdots + a_n)/2 \rfloor \le s-1.$$
	Therefore, 
	$$2(a_2 + a_3 + \cdots  + a_t) + (a_{t+1} + \cdots + a_n) \le 2s-1.$$
	By Lemma \ref{radical_colon}, there exists a minimal generator $P = x^\b$ of $I^{(s)}$ such that $x_1 = \sqrt{P/\gcd(P,x^\a)}$. Thus, $b_1 > a_1$ and $b_i \le a_i$ for all $i = 2, ..., n$. By Lemma \ref{partial_deg_bound_2}, 
	$$a_1 < b_1 \le \sum_{i\in N(1)} b_i \le \sum_{i\in N(1)}a_i.$$	Therefore,
	$$|\a| < 2(a_2 + a_3 + \cdots  + a_t) + (a_{t+1} + \cdots + a_n) \le 2s-1,$$ 
	a contradiction.
\end{proof} 

\begin{rem} In general, $\Delta_\a(J) \neq \Delta_\a(I^s)$ when $|\a| \le 2s-2$. For an example, let $I = (x_1x_2,x_1x_3,x_2x_3,x_3x_4,x_4x_5)$ and $x^\a = x_1x_2x_3x_4$. Then $x_5 \in \sqrt{I^{(3)} : x^\a}$, but $x_5 \notin \sqrt{I^3:x^\a}.$
\end{rem}

Steps 3, 4, 5 are done in the next three lemmas.

\begin{lem}\label{lem_cone} Let $\a \in \NN^n$ be an exponent such that $\ord_I(x^\a) = r < s$ and $|\a| \ge 2s-1$. Let $x^\a = MN$ be a good representation of $x^\a$. Let $i$ be a critical index of $N$ such that $x_i^2 |N$, then $\Delta_\a(I^s)$ is a cone over $i$.
\end{lem}
\begin{proof} For simplicity of notation, assume that $i = 1$ is the critical index of $N$ such that $x_1^2|N$. We will prove by induction on $s$ that $\Delta_\a(I^s)$ is a cone over $1$. For the base case $s \le 2$, since $a_1 \ge 2$, by Remark \ref{rem_extremal_set}, $\Delta_\a(I)$ is a cone over $1$.

Assume that $s \ge 3$. Since $\deg N \ge 2(s-r) - 1$, $1$ is a critical index of $N$, and $|\supp N| \le 2$, we deduce that $b_1 \ge (s-r)$ where $b_1$ is the exponent of $1$ in $N$. Since $\sqrt{I^s:x^\a} \supseteq \sqrt{I^{s-r}:N}$, we deduce that $x_j \in \sqrt{I^s:x^\a}$ for all $j \in N(1)$. 

Assume that $N(1) = \{2, ..., t\}$. Since $\alpha(G)=2$, $\{t+1, ..., n\}$ forms a clique in $G$. Since $a_1 \le s-1$, $r \ge 1$. Write $M = e_1 \cdots e_r$ where $e_i$ are edges in $G$. Note that none of the edges $e_i$ is of the form $x_i x_j$ where $i,j \in \{2, ..., t\}$, as otherwise $x_1^2x_ix_j \in I^2$, and thus $x^\a \in I^{r+1}$, a contradiction. 

Assume by contradiction that $\Delta_\a(I^s)$ is not a cone over $1$. By Lemma \ref{cone}, there exists a minimal generator $f \in \sqrt{I^s:x^\a}$ such that $x_1 | f$. Since $\alpha(G) = 2$, $\deg f \le 2$. There are two cases: 

{\noindent \bf Case 1.} $\deg f = 2$. Since $x_j \in \sqrt{I^s:x^\a}$, for all $j \in N(1)$, we may assume that $f = x_1 x_{t+1}$. By Lemma \ref{criterion_in_power}, $\sum_{i \neq 1, t+1} a_i \ge s$. 

If one of the edges $e_i$ is of the form $x_1x_i$ or $x_{t+1}x_j$, say $e_1$. Let $x^\b = e_2 \cdots e_r N$. Then $\sum_{i \neq 1, t+1} b_i \ge s-1$. By Lemma \ref{criterion_in_power}, $f\in \sqrt{I^{s-1}:x^\b}$. By induction $\Delta_\b(I^{s-1})$ is a cone over $1$, thus $x_{t+1} \in \sqrt{I^{s-1}:x^\b} \subseteq \sqrt{I^s:x^\a}$, which is a contradiction.

If $\supp M \subseteq \{t+2, ..., n\}$, then $\sum_{j\in N(t+1)} a_j \ge a_{t+2} + \cdots + a_n = \sum_{j\neq 1, t+1} a_j \ge s$. By Lemma \ref{criterion_in_power}, $x_{t+1} \in \sqrt{I^s:x^\a}$, which is a contradiction.

Thus we may assume that one of the edges $e_i$ is of the form $x_ix_u$ where $i \in \{2, ..., t\}$, $u\in \{t+2, ..., n\}$, say $e_1=x_i x_{t+2}$. This implies that $\supp N \subseteq \{1, t+2\}$, as otherwise, say $\supp N = \{1, u\}$ for some $u \in \{t+1, t+3,  ..., n\}$, then 
$$(x_1x_u)e_1 = (x_1x_i)(x_ux_{t+2}) \in I^2 \implies x^\a \in I^{r+1}, \text{ a contradiction}.$$
Thus $a_{t+1} = 0$. Let $A = N(t+1) \cap \supp \a$ and $B = ([n] \setminus N(t+1)) \cap \supp \a$. Denote $a = \sum_{i \in A} a_i$ and $b = \sum_{j \in B \setminus \{1\} } a_j$. Then we have 
$$a + b = \sum_{j \neq 1,t+1} a_j  \ge s.$$
Since $t+2 \in N(t+1)$, $a \ge 1$. Denote 
$$h = \prod_{j \notin N[t+1]} x_j^{a_j} = x_1^{a_1} \prod_{j\in B\setminus \{1\} } x_j^{a_j}.$$
If $a_1 \ge b$, by Lemma \ref{clique_term}, $\ord_I(h) = b$. Thus, 
$$\sum_{i \in N(t+1)}a_i + \ord_I(h) = a+ b \ge s.$$
If $a_1 < b$. By Lemma \ref{clique_term}, $\ord_I(h) = \lfloor (b + a_1)/2 \rfloor$. We have $|\a| = a+ b + a_1 \ge 2s-1$. Thus 
$$\sum_{i \in N(t+1)}a_i + \ord_I(h)  = a + \lfloor (b+a_1)/2 \rfloor \ge a + \lfloor (2s-1-a)/2 \rfloor \ge s$$
as $a \ge 1$. By Lemma \ref{criterion_in_power}, in either cases, $x_{t+1} \in \sqrt{I^{s}:x^\a}$, a contradiction.

{\noindent \bf Case 2.} $\deg f = 1$, i.e., $f = x_1$. If one of the edge $e_i$ is of the form $x_1x_i$, say $e_1$, let $x^\b = e_2 \cdots e_r N$, then the condition that $x_1 \in \sqrt{I^s:x^\a}$ also implies that $x_1 \in \sqrt{I^{s-1}:x^\b}$ which is a contradiction, as by induction $\Delta_\b(I^{s-1})$ is a cone over $1$. 

By Lemma \ref{radical_colon}, there exists a minimal generator $P$ of $I^s$ such that 
$$x_1 = \sqrt{P/\gcd(P,x^\a)}.$$
Since $x_1^2 | x^\a$, this implies that $x_1^3|P$. Therefore, $\sum_{i\in N(1)} a_i \ge 3$. Thus three of the edges $e_i$ is of the form $x_i x_u$ for some $i \in \{2, ..., t\}$, $u \in \{t+1, ..., n\}$. We may assume that $e_1 = x_2x_{t+1}$. This implies that all edges of the form $x_ix_u$ must be of the form $x_ix_{t+1}$, as otherwise 
$$(x_2x_{t+1})(x_ix_u) x_1^2  = (x_1x_2)(x_1x_i)(x_{t+1}x_u) \in I^3 \implies x^\a \in I^{r+1}.$$
Since three of the edges $e_i$ is of the form $x_ix_{t+1}$, we may assume that $e_2 = x_ix_{t+1}$. We claim that none of the edges $e_i$ is of the form $x_ux_v$ with $u,v\in \{t+2, ..., n\}$, as otherwise 
$$(x_2x_{t+1})(x_ix_{t+1})(x_ux_v)x_1^2 = (x_1x_2)(x_1x_i)(x_{t+1}x_u)(x_{t+1}x_v)\in I^4 \implies x^\a \in I^{r+1}.$$
In other words, all $e_i$ is of the form $x_ix_{t+1}$ for $i \in N(t+1)$. Assume that $s_1$ edges have $i \in \{2, ..., t\}$ and $s_2$ edges have $i \in \{t+2, ...,n\}$. Let $h = \prod_{j=t+1}^{n}x_j^{a_j}$. We have 
$$\sum_{i\in N(1)} a_i = s_1, \, \sum_{j=t+2}^n a_j = s_2, \,  a_{t+1} \ge s_1 + s_2.$$
By Lemma \ref{clique_term} and Lemma \ref{criterion_in_power}, 
$$r = s_1 + s_2 = \sum_{i\in N(1)} a_i + \ord_I(h) \ge s,$$
which is a contradiction.
\end{proof}

\begin{lem}\label{lem_cone_2} Let $\a \in \NN^n$ be an extremal exponent of $I^s$ such that $|\a| \ge 2s$. Then $\girth \Delta = 4$ and $\Delta_\a(I^s)$ is the empty complex.
\end{lem}
\begin{proof} Let $x^\a = MN$ be a good representation of $x^\a$. By Lemma \ref{lem_cone}, $N$ is squarefree. Since $x^\a \notin I^s$, $\deg M \le 2s-2$. Since $|\a| \ge 2s$ and $\dim \Delta =1$, we must have $\deg N = 2$, $\deg M = 2s-2$, and $|\a| = 2s$. Assume that $N = x_1 x_{t+1}$, $N(1) = \{2, ..., t\}$, and $M = e_1 \cdots e_{s-1}$. 

Since $N = x_1 x_{t+1}$, none of the $e_i$ is of the form $x_i x_v$ where $i \in \{2, ..., t\}$ and $v \in \{t+2, ..., n\}$. Assume that there are $s_1$ edges of the form $1i$, $s_2$ edges of the form $ij$, $s_3$ edges of the form $k (t+1)$, and $s_4$ edges of the form $uv$ where $i,j \in \{2, ..., t\}$, $u,v \in \{t+2, ..., n\}$, $k \in N(t+1)$. Then $s_1 + s_2+s_3 + s_4 = s-1$. 

Since $\sqrt{I^s:x^\a} \supseteq \sqrt{I:(x_1x_{t+1})}$, $x_j \in \sqrt{I^s:x^\a}$ for all $j \neq 1, t+1$. Since $\Delta_\a(I^s)$ is not a cone over $x_j$ for $j = 1, t+1$, we deduce that $x_1 x_{t+1} \in \sqrt{I^s:x^\a}$. By Lemma \ref{criterion_in_power}, 
$$\sum_{i\neq 1, t+1} a_i= s_1 + s_3 + 2(s_2+s_4) = s-1 + (s_2+s_4) \ge s.$$
This implies that either $s_2$ or $s_4$ is positive. Assume that $s_2 > 0$. Then $x_1 \in \sqrt{I^s:x^\a}$ as $x_1^2 x_ix_j \in I^2$. For $\Delta_\a(I^s)$ not be a cone over $x_{t+1}$, we must have $x_{t+1} \in \sqrt{I^s:x^\a}$. In other words, $\Delta_\a(I^s)$ is the empty complex. 

It remains to show that $\girth \Delta = 4$. Note that, for all edges $e_k$, $k = 1, ..., s-1$ of the form $e_k = x_ix_j$, where $i,j \in \{2, ..., t\}$, then $x_ix_{t+1}, x_j x_{t+1} \notin I$, as otherwise 
$$N e_k = x_1x_{t+1}x_ix_{j} = (x_1x_i)(x_{t+1}x_j)\in I^2 \implies x^\a \in I^s.$$
Since $s_2 > 0$, we may assume that $e_1 = x_2x_3$. We claim that for all edge $e_k$ of the form $e_k = x_1x_i$, then $x_ix_{t+1} \notin I$. Otherwise,
$$Ne_1e_k = x_1x_{t+1} x_2x_3 x_1 x_i = (x_1x_2)(x_1x_3)(x_{t+1}x_i) \in I^3 \implies x^\a \in I^s.$$
Thus
$$\sum_{j \in N(t+1)} a_j= s_3 + 2s_4.$$
Let $h = \prod_{j \notin N[t+1]} x_j^{a_j}$. Then 
$$\deg h = |\a| - \sum_{j\in N[t+1]} a_j = (2s - 2s_3 - 2s_4 - 1).$$
By the reasoning in the previous paragraph, we see that $h$ is the products of edges of the first and second kind and $x_1$. In particular, $\ord_I(h) = \lfloor \deg h / 2\rfloor$. By Lemma \ref{criterion_in_power}, we have 
$$\sum_{j \in N(t+1)}a_j  + \ord_I(h)  = s_3 + 2s_4 + s-s_3 -s_4 - 1 = s + s_4 - 1 \ge s.$$
Thus $s_4 > 0$. We may assume that $e_2 = x_ux_v$ for some $u,v\in \{t+2, ...,n\}$. We claim that $2u, 3u, 2v,3v \notin G$. Assume by contrary that one of them belongs to $G$, say $2u \in G$, then 
$$Ne_1e_2 = x_1x_{t+1} x_2x_3 x_ux_v = (x_1x_3)(x_{t+1}x_v)(x_2x_u) \in I^3 \implies x^\a \in I^s.$$
Thus $2u3v$ form a $4$-cycle in $\Delta$, and so $\girth \Delta = 4$ as required.

The case where $s_4 > 0$ can be done in a similar manner, as then $x_{t+1} \in \sqrt{I^s:x^\a}$. We further deduce that $x_1 \in \sqrt{I^s:x^\a}$ and $s_2 > 0$.
\end{proof}

\begin{lem}\label{lem_cone_3} Assume that $\girth \Delta \ge 5$. Let $\a$ be an extremal exponent of $I^s$ such that $|\a| = 2s-1$. Then $\Delta_\a(I^s)$ has homology in degree $-1$ only.
\end{lem}
\begin{proof} By Lemma \ref{lem_cone}, we deduce that $x^\a = M x_1$, where $M =e_1 \cdots e_{s-1}$ for some edges $e_i$ of $G$. We keep the notation as in the previous proof. Namely, $N(1) = \{2, ..., t\}$. Since $\sqrt{I^s:x^\a} \supseteq \sqrt{I:x_1}$, $x_i \in \sqrt{I^s:x^\a}$ for all $i \in N(1)$.

If one of the $e_i$ is of the form $x_ix_u$ for $i \in \{2, ..., t\}$, $u \in \{t+1, ..., n\}$, we can write
$$x^\a = (x_1x_i)M_1 x_u.$$
Thus $\sqrt{I^s:x^\a} \supseteq \sqrt{I:x_u}$. In particular, $x_j \in \sqrt{I^s:x^\a}$ for all $j \in N(1) \cup N(u)$. Therefore, $\Delta_\a(I^s)$ is the empty complex.

If one of the edge $e_i$ is of the form $x_ix_j$, for $i, j \in \{2, ..., t\}$. Then since $x_1x_ix_j = (x_1x_i)x_j = (x_1x_j)x_i$, we deduce that $x_k \in \sqrt{I^s:x^\a}$, for all $k \in N(i) \cup N(j)$. Now since $\girth G \ge 5$ and $\{t+1, ..., n\}$ is a clique, $|\{t+1, ...,n\} \setminus (N(i) \cup N(j)| \le 1$. By Lemma \ref{cone}, $\Delta_\a(I^s)$ is a subsimplicial complex of a singleton, thus can only have homology in degree $-1$ only.

Thus, we may assume that all edges $e_i$ are of the form $x_1x_i$ or $x_u x_v$ where $i \in \{2, ..., t\}$ and $u,v\in \{t+1, ..., n\}$. Let $s_1$ and $s_2$ be the number of edges of the first and second type respectively. Then $s_1 + s_2 = s-1$. Let $h = \prod_{j\notin N[1]} x_j^{a_j}$, then $h$ is the product of edges of the second type. Therefore, $\ord_I(h) = s_2$. Thus 
$$\sum_{i\in N(1)}a_i + \ord_I(h) = s_1 + s_2 = s-1.$$
By Lemma \ref{criterion_in_power}, $x_1 \notin \sqrt{I^s:x^\a}.$ By Lemma \ref{cone} and the fact that $\Delta_\a(I^s)$ is not a cone over $1$, we assume that $x_1x_{t+1}$ is a minimal generator of $\sqrt{I^s:x^\a}$. By Lemma \ref{criterion_in_power}, $\sum_{i \neq 1,t+1} a_i \ge s$. Thus $s_2 > 0$. If one of the edges of the second form $x_ux_v$ is such that $u,v\in \{t+2, ..., n\}$, then $x_{t+1} \in \sqrt{I^s:x^\a}$, since $x_{t+1}^2x_ux_v \in I^2$. Thus all edges of the second type must be of the form $x_{t+1}x_u$ for some $u \in \{t+2, ..., n\}$. But then $x^\a = \prod(x_1x_i) \prod (x_{t+1}x_u) x_1$, thus
$$\sum_{i\neq 1, t+1} a_i= s_1 + s_2 = s-1 < s,$$
which is a contradiction.
\end{proof}

We are now ready for the proof of Theorem \ref{dim1}. We separate two lemmas for the case $\girth \Delta = 4$ and the case $\girth \Delta \ge 5$.

\begin{lem}\label{reg_girth_4} Let $\Delta$ be one-dimensional simplicial complex. Assume that $\girth \Delta = 4$. Then for all $s$ and all intermediate ideal $J$ in $\Inter(I^s, I^{(s)})$ we have
$$\reg J = 2s+1.$$
\end{lem}
\begin{proof} Since $\girth \Delta =4$, the graph $G$ which is the complement of $\Delta$ has induced matching number at least $2$. By Lemma \ref{reg_lower_bound}, $\reg J \ge 2s+ 1$. 

Fix $s \ge 2$. Let $(\a,i)$ be an extremal exponent of $S/J$. It suffices to prove that $|\a| + i \le 2s$. Since $\Delta_\a(J)$ is a subsimplicial complex of $\Delta$,  $i \le 2$. If $i = 2$, by Lemma \ref{h1_bounds}, $|\a| \le 2s-2$. Thus $|\a| + i \le 2s$. 

Note that if $|\a| \le 2s-2$, then $|\a| + i \le 2s$. Thus, we may assume that $|\a| \ge 2s-1$. By Lemma \ref{intermediate_reduction}, $\Delta_\a(J) = \Delta_\a(I^s)$. By Lemma \ref{lem_cone}, $|\a| \le 2s$. Furthermore, by Lemma \ref{lem_cone_2}, if $|\a| = 2s$, then $i = 0$. This implies that $|\a| + i \le 2s$, as required.
\end{proof}

\begin{lem}\label{reg_girth_5} Let $\Delta$ be a one-dimensional simplicial complex. Assume that $\girth \Delta \ge 5$. Then for all $s \ge 2$ and all intermediate ideal $J$ in $\Inter(I^s,I^{(s)})$, we have 
$$\reg J = 2s.$$
\end{lem}
\begin{proof} By degree reason, $\reg S/J \ge 2s-1$. Fix $s \ge 2$. Let $(\a,i)$ be an extremal exponent of $S/J$. It suffices to prove that $|\a| +i \le 2s-1$. Since $\Delta_\a(J)$ is a subsimplicial complex of $\Delta$,  $i \le 2$. If $i = 2$, by Lemma \ref{h1_bounds}, $|\a| \le 2s-3$. Thus $|\a| + i \le 2s-1$. 

Thus, we may assume that $i \le 1$. If $|\a| \le 2s-2$, then $|\a| + i \le 2s-1$. Thus, it remains to consider the case where $|\a| \ge 2s-1$. By Lemma \ref{intermediate_reduction}, $\Delta_\a(J) = \Delta_\a(I^s)$. 

By Lemma \ref{lem_cone}, Lemma \ref{lem_cone_2}, and the assumption that $\girth \Delta \ge 5$, we must have $|\a| = 2s-1$. By Lemma \ref{lem_cone_3}, $i = 0$. In all cases, $|\a| + i \le 2s-1$, as required.
\end{proof}

\begin{proof}[Proof of Theorem \ref{dim1}] Follows from Lemmas \ref{reg_girth_3}, \ref{reg_girth_4}, and \ref{reg_girth_5}.
\end{proof}

\section{ Graphs on small number of vertices }\label{sec_small_graphs}

In this section, we provide an example to illustrate our procedure for computing regularity of intermediate ideals $J \in \Inter(I^s, I^{(s)})$ for edge ideals of graphs with $\alpha(G) > 2$. More precisely, we prove

\begin{thm}\label{small_graphs} Let $G$ be a graph supported on $n \le 5$ vertices. For each $s \ge 2$, let $J \in \Inter(I^s,I^{(s)})$ be an intermediate ideal. Then, $\reg J= \reg I^s = \reg I^{(s)}.$
\end{thm}
By Theorem \ref{dim1}, it remains to consider non-bipartite graphs $G$ with $\alpha(G)  > 2$. There are precisely four such graphs, listed below 
\begin{align*}
    (x_1x_2,x_1x_3,x_2x_3,x_2x_4,x_3x_5),&(x_1x_2,x_1x_3,x_2x_3,x_1x_4,x_1x_5),\\
    (x_1x_2,x_1x_3,x_2x_4,x_3x_4,x_3x_5),&(x_1x_2,x_1x_3,x_2x_4,x_3x_4,x_3x_5,x_2x_5,x_3x_4).
\end{align*}
The regularity of powers and symbolic powers of these graphs are known from \cite{B}, \cite{F}. Nonetheless, we will give an explicit calculation for all intermediate ideals for the first graph to illustrate techniques used in the paper. The other cases can be done in similar manner and are left as an exercise for interested readers.
\begin{lem} Let $I = (x_1x_2,x_1x_3,x_2x_3,x_2x_4,x_3x_5) \subset K[x_1,...,x_5]$. Let $J$ be an intermediate ideal in $\Inter(I^s,I^{(s)})$. Then 
$$\reg J = 2s.$$
\end{lem}
\begin{proof} By degree reason, $\reg J \ge 2s$. Let $(\a,i) \in \NN^5 \times \NN$ be an extremal exponent of $J$. It suffices to prove that $|\a| +i \le 2s-1$. From the definition, we have 
\begin{align*}
    \ord_I(x_1^{a_1}x_3^{a_3}x_5^{a_5}) & = \min(a_3,a_1+a_5)\\
    \ord_I(x_1^{a_1}x_2^{a_2}x_4^{a_4}) & =  \min(a_2,a_1+a_4). 
\end{align*}

Note that $N(1) = N(\{4,5\}) = N(\{1,4,5\}) = \{2,3\}$, and any monomials supported on the complement of $N[1]$, $N[\{4,5\}]$ have order $0$. By Lemma \ref{criterion_in_power} and Lemma \ref{criterion_in_sym}, we have

\begin{align}
    x_1 \in \sqrt{J:x^\a} & \Leftrightarrow a_2 + a_3 \ge s \Leftrightarrow x_1 x_4x_5 \in \sqrt{J:x^\a} \label{eq1}\\ 
    x_2 \in \sqrt{J:x^\a} & \Leftrightarrow a_1 + a_3 + a_4 \ge s \Leftrightarrow x_2 x_5 \in \sqrt{J:x^\a} \label{eq2}\\
    x_3 \in \sqrt{J:x^\a} & \Leftrightarrow a_1 + a_2 + a_5 \ge s \Leftrightarrow x_3 x_4 \in \sqrt{J:x^\a} \label{eq3} \\
    x_4 \in \sqrt{J:x^\a} & \Leftrightarrow a_2 + \min(a_3,a_1+a_5) \ge s\label{eq4} \\
    x_5 \in \sqrt{J:x^\a} & \Leftrightarrow a_3 + \min(a_2,a_1+a_4) \ge s.\label{eq5}
\end{align}
In particular, 
\begin{equation}
    x_1x_4, x_1x_5, x_1x_4x_5, x_2x_5, \text{ and } x_3x_4 \text{ cannot be minimal in } \sqrt{J:x^\a}.\label{eq6} 
\end{equation}

Since $\Delta_\a(J)$ is a subsimplicial complex of 
$$\D(I) = \{\{1,4,5\},\{2,5\},\{3,4\}\},$$
$i = 2$ if and only if $14,45,51$ are facets of $\Delta_\a(J)$. This implies that $x_1x_4x_5$ is minimal in $\sqrt{J:x^\a}$, which is not the case. Thus $i \le 1$. If $|\a| \le 2s-2$, $|\a| +i \le 2s-1$ as required. Now assume that $|\a| \ge 2s-1$. By Remark \ref{rem_extremal_set}, $a_i \le s-1$.

Since $(a_1 + a_2 + a_5) + (a_1 + a_3 + a_4) = a_1 + |\a| \ge 2s - 1$, one of the term is at least $s$. By \eqref{eq2}, \eqref{eq3}, either $x_2$ or $x_3$ belongs to $\sqrt{J:x^\a}$. In other words, 
\begin{equation}
    x_2x_3 \text{ cannot be minimal in } \sqrt{J:x^\a}.\label{eq7}
\end{equation}

First, assume that $\supp \a \subseteq \{1,2,3\}$. Since $a_i \le s-1$, and $|\a| \ge 2s-1$, we have $0<a_1 < a_2 + a_3, 0<a_2 < a_1 + a_3, 0<a_3 < a_1 + a_2$. By Lemma \ref{clique_term},
$$\ord_I(x^\a) = \lfloor |\a|/2 \rfloor, \, \ord_I(x^\a/x_i) = \lfloor (|\a| - 1)/2 \rfloor,$$
for all $i = 1, 2, 3$. Since $x^\a \notin J$, $|\a| =2s-1$, and $x^\a/x_i \in I^{s-1}$ for $i = 1, 2, 3$. In particular, $x_i \in \sqrt{J:x^\a}$ for all $i \in N(1) \cup N(2) \cup N(3)$. Thus $\Delta_\a(J)$ is the empty complex.

Now assume that $\supp \a \cap \{4,5\} \neq \emptyset$. By symmetry, we may assume that $4 \in \supp \a$. By Remark \ref{rem_mingens_degree_complex}, there exists a minimal generator $f$ of $\sqrt{J:x^\a}$ such that $x_4 | f$. By \eqref{eq6}, we deduce that $f$ must be one of the following $x_2x_4,x_4x_5$ or $x_4$. 

{\noindent \bf Case 1.} $f=x_2x_4$ is minimal in $\sqrt{J:x^\a}$. Then $x_2, x_4 \notin \sqrt{J:x^\a}$. By \eqref{eq2}, \eqref{eq4},
$$a_1 + a_3 + a_4 < s,\, a_2 + \min(a_3, a_1 + a_5) < s.$$
Since $|\a| \ge 2s - 1$, $a_2 + a_5 \ge s$. Thus, $a_3 < a_1 + a_5$ and $a_2 + a_3 < s$. By \eqref{eq3}, $x_3 \in \sqrt{J:x^\a}$. Since $a_2 + a_5 \ge s$ and $a_2 \le s-1$, $a_5 > 0$. Thus $5 \in \supp \a$. By Remark \ref{rem_mingens_degree_complex}, \eqref{eq6}, and the fact that $x_3 \in \sqrt{J:x^\a}$, either $x_5$ or $x_4x_5$ must be minimal in $\sqrt{J:x^\a}$. In either cases, we must have $x_4x_5 \in \sqrt{J:x^\a}$. By \eqref{eq1}, $a_2 + a_3 \ge s$, a contradiction.

{\noindent \bf Case 2.} $f = x_4x_5$ is minimal in $\sqrt{J:x^\a}$. Then $x_4, x_5 \notin \sqrt{J:x^\a}$. By \eqref{eq1}, \eqref{eq4}, we have $x_1 \in \sqrt{J:x^\a}$, and 
$$a_2 + a_3 \ge s, \, a_2 + \min(a_3, a_1+a_5) < s.$$
Thus, $a_3 > a_1 + a_5$ and $a_2 + a_1 + a_5 < s$. Since $|\a| \ge 2s - 1$, $a_3 + a_4 \ge s$. Thus, 
$$a_3 + \min(a_2, a_1 + a_4) \ge s \implies x_5 \in \sqrt{J:x^\a} \text{ by } \eqref{eq5},$$
a contradiction.

Therefore, $x_4 \in \sqrt{J:x^\a}$. By \eqref{eq1},
$$a_2 + a_3 \ge s, \text{ and } x_1 \in \sqrt{J:x^\a}.$$

Since $a_2, a_3 \le s-1$, $a_2, a_3 \ge 1$. By Remark \ref{rem_mingens_degree_complex}, \eqref{eq6}, \eqref{eq7}, and the fact $x_1, x_4 \in \sqrt{J:x^\a}$, we must have $x_2 \in \sqrt{J:x^\a}$. Thus $a_1 + a_3 + a_4 \ge s$. Therefore,
$$a_3 + \min(a_2, a_1+a_4) \ge s \implies x_5 \in \sqrt{J:x^\a} \text{ by } \eqref{eq5}.$$
Since $3 \in \supp \a$, by Remark \ref{rem_mingens_degree_complex}, $x_3 \in \sqrt{J:x^\a}$. Therefore, $\Delta_\a(J)$ is the empty complex. We have the following 
$$a_1 + a_2 + a_5 \ge s, a_2 + a_3 \ge s, a_1 + a_3 + a_4 \ge s.$$
Note that $x^\a \in I^r$ for 
$$r = \min(a_4,a_2) + \min(a_1+a_5,a_3)$$
thus $a_4 < a_2$. Similarly, $a_5 < a_3$. Write $x^\a = (x_2x_4)^{a_4}(x_3x_5)^{a_5} h$. Then $h = x^\b$ with $b_1 = a_1, b_2 = a_2 - a_4, b_3 = a_3 - a_5.$ Since $x^\a \notin J$, $a_4 + a_5 + \ord_I(h) < s$. Since, $123$ is a triangle, and $a_4 + a_5 + b_i + b_j \ge s$, for all $i \neq j$ in $\{1,2,3\}$, by Lemma \ref{clique_term} we must have $\ord_I(h) = \lfloor \deg h / 2 \rfloor$. Thus 
$$a_4 + a_5 + \lfloor |\a|/2 \rfloor - a_4 - a_5 < s.$$
In particular, $|\a| < 2s$, as required. Thus $\reg J = 2s$.
\end{proof}

\subsection*{Acknowledgment} Nguyen Cong Minh is partially supported by the National Foundation for Science and Technology Development (Vietnam).

\end{document}